%% file: TransientGrowth_Revised_arXiv.tex
\documentclass[journal,onecolumn]{IEEEtran}

\IEEEoverridecommandlockouts                                 

\usepackage{subfig}
\usepackage{epic,eepic,units}
\usepackage{psfrag}
\usepackage{latexsym}
\usepackage{longtable}
\usepackage{mathrsfs}
\usepackage{bigstrut}
\usepackage{graphicx}
\usepackage{pifont}
\usepackage{amsmath}
\usepackage{setspace}
\usepackage{mathtools}
\usepackage{mathcomp}
\usepackage{mathdots}
\usepackage{rotating}
\usepackage{array}
\usepackage{multirow}
\usepackage{multicol}
\usepackage{color}
\usepackage{url}
\usepackage{algorithm}
\usepackage{algorithmic}

\usepackage{cite,amssymb}
\usepackage[dvipsnames]{xcolor}

\usepackage{euscript,yfonts,dsfont,bbm,amstext,wasysym,pdfsync,amsfonts,arydshln,framed}
\usepackage{verbatim}
\usepackage{marginnote}

\input{commands}
\input{commandsHesam}

\newcommand{\bbR}{\mathbb{R}}

\newcommand{\non}{\nonumber}
\newcommand{\ds}{\displaystyle}

\newcommand{\mrd}{\mathrm{d}}
\newcommand{\mre}{\mathrm{e}}

\newcommand{\na}{\mathrm{na}}
\newcommand{\hb}{\mathrm{hb}}

\newcommand{\bz}{{\bar{z}\,}}

\usepackage{tikz}
\usetikzlibrary{positioning,shapes,arrows,calc,matrix}

\tikzset{
	block/.style = {draw, rectangle, 
		minimum height=0.5cm, 
		minimum width=1cm},
	input/.style = {coordinate,node distance=2.5cm},
	output/.style = {coordinate,node distance=2.5cm},
	arrow/.style={draw, -latex,node distance=2cm},
	sum/.style = {draw, circle, node distance=1cm}
}

\begin{document}

\title{\huge
Transient growth of  accelerated optimization algorithms}

\author{Hesameddin Mohammadi, Samantha Samuelson, and Mihailo R.\ Jovanovi\'c
\thanks{Financial support from the National Science Foundation under awards ECCS-1708906 and ECCS-1809833 is gratefully acknowledged.}
\thanks{The authors are with the Ming Hsieh Department of Electrical and Computer Engineering, University of Southern California, Los Angeles, CA 90089. E-mails: (\{hesamedm, sasamuel, mihailo\}@usc.edu).}
}

\maketitle
\pagestyle{empty}

	\vspace*{-3ex}
\begin{abstract} 
Optimization algorithms are increasingly being used in applications with limited time budgets. In many real-time and embedded scenarios, only a few iterations can be performed and traditional convergence metrics cannot be used to evaluate performance in these non-asymptotic regimes. In this paper, we examine the transient behavior of accelerated first-order optimization algorithms. For convex quadratic problems, we employ tools from linear systems theory to show that transient growth arises from the presence of non-normal dynamics. We identify the existence of modes that yield an algebraic growth in early iterations and quantify the transient excursion from the optimal solution caused by these modes. For strongly convex smooth optimization problems, we utilize the theory of integral quadratic constraints (IQCs) to establish an upper bound on the magnitude of the transient response of Nesterov's accelerated algorithm. We show that both the Euclidean distance between the optimization variable and the global minimizer and the rise time to the transient peak are proportional to the square root of the condition number of the problem. Finally, for problems with large condition numbers, we demonstrate tightness of the bounds that we derive up to constant factors.
\end{abstract}

	\vspace*{-2ex}
\begin{IEEEkeywords}
	Convex optimization, first-order optimization algorithms, heavy-ball method, integral quadratic constraints, Nesterov's accelerated method, non-asymptotic behavior, non-normal matrices, transient growth.
\end{IEEEkeywords}

	\vspace*{-2ex}
\section{Introduction}
\label{se.intro}

First-order optimization algorithms are widely used in a variety of fields including statistics, signal and image processing, control, and machine learning~\cite{botcun05,becteb09,nes13,honrazluopan15,botcurnoc18,linfarjovTAC13admm,mogjovTCNS18,zarmohdhigeojovTAC20}.  Acceleration is often utilized as a means to achieve a faster rate of convergence relative to gradient descent while maintaining low per-iteration complexity.  There is a vast literature focusing on the convergence properties of accelerated algorithms for different stepsize rules and acceleration parameters, including~\cite{sutmardahhin13,pol64,nes83,nes18book}. There is also a growing body of work which investigates robustness of accelerated algorithms to various types of uncertainty~\cite{devgliNes14,hules17,mohrazjovCDC18,mohrazjovACC19,mohrazjovTAC21,micschebe20,povli21}. These studies demonstrate that acceleration increases sensitivity to \mbox{uncertainty in gradient evaluation.}

In addition to deterioration of robustness in the face of uncertainty, asymptotically stable accelerated algorithms may also exhibit undesirable transient behavior~\cite{doncan15}. This is in contrast to gradient descent which is a contraction for strongly convex problems with suitable stepsize~\cite{ber15book}. In real-time optimization and in applications with limited time budgets, the transient growth can limit the appeal of accelerated methods. In addition, first-order algorithms are often used as a building block in multi-stage optimization including ADMM~\cite{ouychelanpas15} and distributed optimization methods~\cite{shilinyuawuyin14}. In these settings, at each stage we can perform only a few iterations of first-order updates on primal or dual variables and transient growth can have a detrimental impact on the performance of the entire algorithm.  This motivates an in-depth study of the behavior of accelerated first-order methods in non-asymptotic regimes.

It is widely recognized that large transients may arise from the presence of resonant modal interactions and non-normality of linear dynamical generators~\cite{treemb05}. Even in the absence of unstable modes, these can induce large transient responses, significantly amplify exogenous disturbances, and trigger departure from nominal operating conditions. For example, in fluid dynamics, such mechanisms can initiate departure from stable laminar flows and trigger \mbox{transition to turbulence~\cite{jovbamJFM05,jovARFM21}.}

\begin{figure}[t]
	\centering
	\begin{tabular}{@{\hspace{0 cm}}r@{\hspace{-0.3 cm}}c}
		\begin{tabular}{c}
			\rotatebox{90}{$\norm{x^t \, - \, x^\star}_2^2$}
		\end{tabular}
		&
		\begin{tabular}{c}
			\includegraphics[width=.22\textwidth]{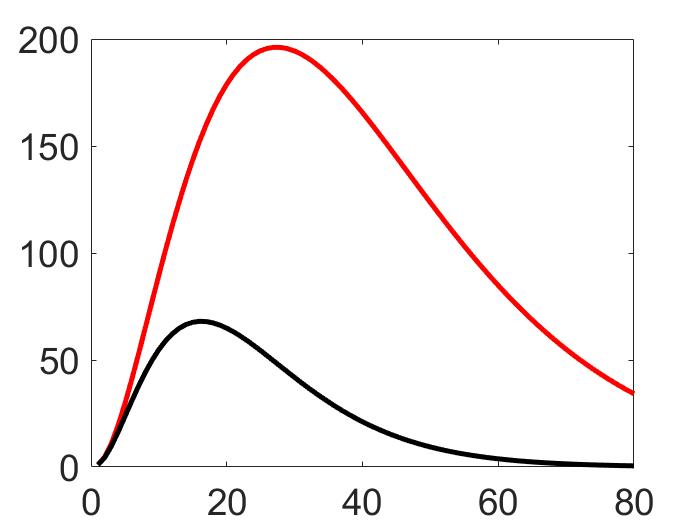}
			\\[-0.0cm]
			{iteration number $t$}
		\end{tabular}
	\end{tabular}
	\caption{Error in the optimization variable for Polyak's heavy-ball (black) and Nesterov's (red) algorithms with the parameters that optimize the convergence rate for a strongly convex quadratic problem with the condition number $10^3$ and a unit norm initial condition with $x^0 \neq x^\star$.}
	\label{fig.motivation}
\end{figure}

In this paper, we consider the  optimization problem
\begin{equation}\label{eq.optProb}
\minimize\limits_x ~f(x)
\end{equation}
where $f$: $\mathbb{R}^n \rightarrow \mathbb{R}$ is a convex and smooth function, and we focus on a class of accelerated first-order algorithms
	\be
	x^{t+2} 
	\, = \,  
	x^{t+1}
	\, + \,  
	\beta
	(
	x^{t+1}
	-  
	x^{t}
	)
	\, - \,  
	\alpha 
	\nabla 
	f
	(
	x^{t+1}
	\, + \,  
	\gamma
	(
	x^{t+1}
	-
	x^{t})
	)	
	\label{eq.1st-order}
	\ee
where $t$ is the iteration index, $\alpha $ is the stepsize, and $\beta$ is the momentum parameter. In particular, we are interested in Nesterov's accelerated and Polyak's heavy-ball methods that correspond to $\gamma=\beta$ and $\gamma=0$, respectively. While these algorithms have faster convergence rates compared to the standard gradient descent ($\gamma=\beta=0$), they may suffer from large transient responses; see Fig.~\ref{fig.motivation} for an illustration. To quantify the transient behavior, we examine the ratio of the largest error in the optimization variable to the initial error. 

For convex quadratic problems,~\eqref{eq.1st-order} can be cast as a linear time-invariant (LTI) system for which modal analysis of the state-transition matrix can be performed. For both accelerated algorithms, we identify non-normal modes that create large transient growth, derive analytical expressions for the state-transition matrices, and establish bounds on the transient response in terms of the convergence rate and the iteration number. We show that both the peak value of the transient response and the rise time to this value increase with the square root of the condition number of the problem. Moreover, for general strongly convex problems, we combine a Lyapunov-based approach with the theory of IQCs to establish an upper bound on the transient response of Nesterov's accelerated algorithm. As for quadratic problems, we demonstrate that this bound scales with the square root of the condition number.
		
This work builds on our recent conference papers~\cite{sammohjovACC20,sammohjovCDC20}. In contrast to these preliminary results, we provide a comprehensive analysis of transient growth of accelerated algorithms for convex quadratic problems and address the important issue of eliminating transient growth of Nesterov's accelerated algorithm with the proper choice of initial conditions.  Adaptive restarting, which was introduced in~\cite{doncan15} to address the oscillatory behavior of Nesterov's accelerated method, provides heuristics for improving transient responses. In~\cite{polsmi19}, the transient growth of second-order systems was studied and a framework for establishing upper bounds was introduced, with a focus on real eigenvalues. The result was applied to the heavy-ball method but was not applicable to quadratic problems in which the dynamical generator may have complex eigenvalues. We account for complex eigenvalues and conduct a thorough analysis for Nesterov's accelerated algorithm as well. Furthermore, for convex quadratic problems, we provide tight upper and lower bounds on transient responses in terms of the condition number and identify the initial condition that induces the largest transient response. Similar results with extensions to the Wasserstein distance have been recently reported in~\cite{cangurzhu19}. Previous work on non-asymptotic bounds for Nesterov's accelerated algorithm includes~\cite{nes04}, where bounds on the objective error in terms of the condition number were provided. However, in contrast to our work, this result introduces a restriction on the initial conditions. Finally, while~\cite{fazribmor18} presents computational bounds we develop analytical bounds on the non-asymptotic value of the estimated optimizer. 

	\vspace*{-1ex}
\section{Convex quadratic problems}
\label{sec.trans}
In this section, we examine transient responses of accelerated algorithms for convex quadratic objective functions,
	\begin{subequations}
	\label{eq.f-Q}
	\beq
	\label{eq.quadObjFunc}
	f(x)  \, = \, \dfrac{1}{2} \, x^T Q \, x
	\eeq
where $Q = Q^T \succeq 0$ is a positive semi-definite matrix. In what follows, we first bring~\eqref{eq.1st-order} into a standard LTI state-space form and then utilize appropriate coordinate transformation to decompose the dynamics into decoupled subsystems. Using this decomposition, we provide analytical expressions for the state-transition matrix and establish sharp bounds on the transient growth and the location of the transient peak for accelerated algorithms. We also examine the influence of initial conditions on transient responses and relegate the proofs to Appendix~\ref{app.lemmaProofs}.

	\vspace*{-1ex}
\subsection{LTI formulation}

The matrix $Q$ admits an eigenvalue decomposition, $Q=V\Lambda V^T$, where $\Lambda$ is the diagonal matrix of eigenvalues with
\beq
	\ba{l}
	L
	\, \DefinedAs \,
	\lambda_1
	\, \ge \, 
	\cdots
	\, \ge \,
	\lambda_r
	\, \AsDefined \,
	m
	\, > \,
	0
	\\[0.1cm]	
	\lambda_{i} 
	\, = \, 
	0~\,\mbox{for}~\, 
	i \, = \, r+1,\ldots,n 
	\ea
	\label{eq.lambdas}
	\eeq
	\end{subequations}
and $V$ is the unitary matrix of the corresponding eigenvectors.  We define the condition number  $\kappa \DefinedAs L/m$  as the ratio of the largest and smallest non-zero eigenvalues of the matrix $Q$. For $f$ in~\eqref{eq.quadObjFunc}, we have  $\nabla f(x) = Q x$, and the change of variables $\hat{x}^t \DefinedAs V^T x^t$ brings dynamics~\eqref{eq.1st-order} to
	\be
	\hat{x}^{t+2} 
	\;=\;
	(I \, - \, \alpha \Lambda ) \, 
	\hat{x}^{t+1}
	\, + \,
	(\beta I \, - \, \gamma \alpha \Lambda ) 
	(
	\hat{x}^{t+1}
	\, - \,  
	\hat{x}^{t}
	).
	\label{eq.NA2}
	\ee
This system can be represented via $n$ decoupled second-order subsystems of the form,
\begin{subequations}
	\label{eq.statespace}
	\begin{align}\label{eq.stateeq}
	\hat{\psi}_i^{t+1} 
	\, = \, 
 	A_i \hat{\psi}_i^t,
	\quad
	\hat{x}_i^{t}
	\, = \,
	C_i \hat{\psi}_i^t
\end{align}
where $\hat{x}_i^{t}$ is the $i$th element of the vector $\hat{x}^t \in \bbR^n$, $\hat{\psi}_i^t \DefinedAs \obt{\hat{x}_i^{t}}{\hat{x}_i^{t+1}}^T$, $C_i \DefinedAs \obt{1}{0}$, and 
\be
	A_i
	\, = \,
	\tbt{0}{1}
	{-(\beta - \gamma \alpha \lambda_i)}
	{1 - \alpha \lambda_i + (\beta- \gamma \alpha \lambda_i)}.
	\label{eq.Ai}
\ee	
\end{subequations}
  
\vspace*{-2ex}
\subsection{Linear convergence of accelerated algorithms}
The minimizers of~\eqref{eq.quadObjFunc} are determined by the null space of the matrix $Q$, $x^\star \in \mathcal{N} (Q)$. The constant parameters $\alpha$ and $\beta$ can be selected to provide stability of subsystems in~\eqref{eq.statespace} for all $\lambda_i \in [m,L]$, and guarantee convergence of $\hat{x}_i^{t}$ to $\hat{x}_i^\star \DefinedAs 0$ with a linear rate determined by the spectral radius $\rho(A_i)<1$. On the other hand, for $i =  r+1, \ldots, n$ the eigenvalues of $A_i$ are $\beta$ and $1$. In this case, the solution to~\eqref{eq.statespace} is given by 
\begin{subequations}
\label{eq.nulDynamics}
\begin{align}\label{eq.nullxhati}
	\hat{x}_i^{t}
	\, = \,
	\dfrac{1 \, - \, \beta^t}{1 \, - \, \beta^{\phantom{t}}} 
	\, 
	(\hat{x}^{1}_i \, - \, \hat{x}^{0}_{0}) 
	\, + \, 
	\hat{x}^{0}_i
\end{align} 
and the steady-state limit of $\hat{x}_i^{t}$,
\begin{align}\label{eq.steadyStatexhatistar}
	\hat{x}_i^\star
	\, \DefinedAs \,
	\dfrac{1}{1 \, - \, \beta}
	\, 
	(\hat{x}^{1}_i \, - \, \hat{x}^{0}_i) \, + \, \hat{x}^{0}_i
\end{align}
\end{subequations}
is achieved with a linear rate $\beta<1$. Thus, the iterates of~\eqref{eq.1st-order} converge to the optimal solution $x^\star = V \hat{x}^\star \in \mathcal{N} (Q)$ with a linear rate $\rho < 1$ and Table~\ref{tab:rates} provides the parameters $\alpha$ and $\beta$ that optimize the convergence rate~\cite[Proposition~1]{lesrecpac16}. 

\begin{table}
	\centering
	\begin{tabular}{|l|ll|l|} 
		\hline 
		& &&
		\\[-.30cm]
		\!\!\! Method  \!\!\! 
		& \!\!\! Optimal parameters \!\!\! & 
		& \!\!\! Linear rate $\rho$ \!\!\!
		\\
		\hline
		\hline
		&&&
		\\[-.25cm] 
		\!\!\! Nesterov \!\!\!
		& 
		\!\!\!
		$ \alpha = \tfrac{4}{3L+m}$
		\!\!\!
		&
		\hspace*{-0.75cm}
		\!\!\!
		$\beta = \tfrac{\sqrt{3\kappa+1}-2}{\sqrt{3\kappa+1}+2}$ 
		\!\!\!
		& 
		\!\!\!
		$1 - \frac{2}{\sqrt{3\kappa+1}}$ 
		\!\!\! 
		\\[0.1cm] 
		\!\!\! Polyak \!\!\!
		& 
		\!\!\!
		$\alpha = \tfrac{4}{(\sqrt{L}+\sqrt{m})^2}$
		\!\!\!
		&
		\hspace*{-0.75cm}
		\!\!\!
		$\beta = \tfrac{( \sqrt{\kappa} - 1 )^2}{ ( \sqrt{\kappa} + 1 )^2} $ 
		\!\!\!
		& 
		\!\!\!
		$1 - \tfrac{2}{\sqrt{\kappa} + 1}$ 
		\!\!\!
		\\[-.35cm]
		& 
		& 
		&
		\\
		\hline
	\end{tabular}
	\caption{Parameters that provide optimal convergence rates for a convex quadratic objective function~\eqref{eq.f-Q} with $\kappa \DefinedAs \Lf / \mf$. 
	}
	\label{tab:rates}
\end{table}
	
	\vspace*{-2ex}
\subsection{Transient growth of accelerated algorithms} 	

In spite of a significant improvement in the rate of convergence, acceleration may deteriorate performance on finite time intervals and lead to large transient responses. This is in contrast to gradient descent which is a contraction~\cite{ber15book}. At any $t$, we are interested in the worst-case ratio of the two norm of the error of the optimization variable $z^t \DefinedAs x^t - x^\star$  to the two norm of the initial condition $\psi^0 - \psi^\star = \obt{ \! (z^0)^T \!\!}{\!\! (z^1)^T \!\!}^T$,
	\beq
	J^2 (t)
	\; \DefinedAs \;
	\sup\limits_{
		\psi^0 \, \neq \, \psi^\star  
	}
	\;
	\dfrac{\norm{x^t \, - \, x^\star}_2^2}{\norm{\psi^0 \, - \, \psi^\star}_2^2}.
	\label{eq.tr}
	\eeq

	\begin{myprop}	
	 For accelerated algorithms applied to convex quadratic problems, $J(t)$ in~\eqref{eq.tr} is determined by
\beq
	J^2 (t)
	\; = \;
	\max
	\left\{
	\displaystyle{\max_{i \, \le \, r}} \; \norm{C_i A_i^t}_2^2, 
	\;
	\beta^{2t}/(1 \, + \, \beta^2)
	\right\}.
	\label{eq.Jquad}
\eeq
	\label{prop.Jquad}
	\end{myprop}
	\vspace*{-2.5ex}
	\begin{IEEEproof}	
Since $V$ is unitary and dynamics~\eqref{eq.statespace} that govern the evolution of each $\hat{x}_i^{t}$ are decoupled, $J (t)$ is determined by
	\beq
	J^2 (t)
	\; = \;
\max\limits_i\sup\limits_{
		\hat{\psi}_i^{0}
		\, \neq \, 
		\hat{\psi}_i^{\star}}
		\;
		\dfrac{(\hat{x}_i^{t} \, - \, \hat{x}_i^{\star})^2}
		{\norm{\hat{\psi}_i^{0} \, - \, \hat{\psi}_i^{\star}}_2^2}
	\label{eq.tr1}
	\eeq
where $\hat{\psi}_i^\star \DefinedAs \obt{\hat{x}_i^\star}{\hat{x}_i^\star}^T$. Furthermore, the mapping from $\hat{\psi}_i^0 - \hat{\psi}_i^\star$ to $\hat{x}_i^{t}-\hat{x}^\star_i$ is given by $\Phi_i (t) \DefinedAs C_i A_i^t$ where the state-transition matrix $A_i^t$ is determined by \mbox{the $t$th power of $A_i$,}
	\beq
	\hat{x}_i^{t} \, - \, \hat{x}_i^{\star}
	\, = \,
	C_i A_i^t (\hat{\psi}_i^0 \, - \, \hat{\psi}_i^\star)
	\, \AsDefined \,
	\Phi_i(t) (\hat{\psi}_i^0 \, - \, \hat{\psi}_i^\star).
	\label{eq.xi-i}
	\eeq 
For $\lambda_i\neq 0$,  $\hat{\psi}_i^{0} - \hat{\psi}_i^{\star}=\hat{\psi}_i^0$ is an arbitrary vector in $\bbR^2$. Thus,
	\beq
	\sup\limits_{
		\hat{\psi}_i^{0}
		\, \neq \, 
		\hat{\psi}_i^{\star} 
	}
	\dfrac{(\hat{x}_i^{t} \, - \, \hat{x}_i^{\star})^2}
	{\norm{\hat{\psi}_i^{0} \, - \, \hat{\psi}_i^{\star}}_2^2} 
	\;=\;
	\norm{C_i A_i^t}_2^2,
	~~
	i \, = \, 1, \ldots, r.
	\label{eq.J1i}
	\eeq
This expression, however, does not hold when $\lambda_i=0$ in~\eqref{eq.statespace} because $\psi_i^{0}-\psi_i^{\star} $ is restricted to a line in $\R^2$. Namely, from~\eqref{eq.nulDynamics}, 
	\beq
	\ba{rcl}
	\hat{x}_i^{t} \, - \, \hat{x}_i^\star 
	& \!\!\! = \!\!\! &
	\dfrac{- \beta^t}{1 \, - \, \beta}  
	\,
	(\hat{x}^{1}_i \, - \, \hat{x}^{0}_{0})
	\\[0.25cm]
	\psi_i^{0} \, - \, \psi_i^{\star} 
	& \!\!\! = \!\!\! &
	\tbo{\hat{x}_i^{0} \, - \, \hat{x}_i^\star}{\hat{x}_i^{1} \, - \, \hat{x}_i^\star}
	\, = \,
	\dfrac{-(\hat{x}_i^{1} \, - \, \hat{x}_i^{0})}{1 \, - \, \beta}
	\tbo{1}{\beta}
	\ea
	\label{eq.zeroDynamics}
	\eeq
which, for any initial condition with $\hat{x}_i^{0}\neq \hat{x}_i^{1}$, leads to
	\beq
	\dfrac{(\hat{x}_i^{t} \, - \, \hat{x}_i^{\star})^2}{\norm{\psi_i^{0} \, - \, \psi_i^{\star}}_2^2}
	\, = \,
	\dfrac{\beta^{2t}}{1 \, + \, \beta^2},
	~~
	i \, = \, r+1, \ldots, n.
	\label{eq.J2i}
	\eeq
Finally, substitution of~\eqref{eq.J1i} and~\eqref{eq.J2i} to~\eqref{eq.tr1} yields~\eqref{eq.Jquad}.
	\end{IEEEproof}

	\vspace*{-2ex}
 \subsection{Analytical expressions for transient response}
\label{subsec.dimReduc}

We next derive analytical expressions for the state-transition matrix $A_i^t$ and the response matrix $\Phi_i (t) = C_i A_i^t$ in~\eqref{eq.statespace}.

	\vsp
	
\begin{mylem}\label{lem.powers}
	Let $\mu_1$ and $\mu_2$ be the eigenvalues of the matrix
	\[
	M 
	\, = \,
	\tbt{0}{1}{a}{b}
	\]
and let $t$ be a positive integer. For $\mu_1 \neq \mu_2$, 
	\begin{align*}	
	M^t 
	\, = \,
	\dfrac{1}{\mu_2 - \mu_1}\,
	\tbt
	{\mu_1 \mu_2 (\mu_1^{t-1} - \mu_2^{t-1})}{\mu_2^t - \mu_1^t}
	{\mu_1 \mu_2 (\mu_1^{t} - \mu_2^{t})}{\mu_2^{t+1} - \mu_1^{t+1}}.
	\end{align*}
Moreover, for $\mu \DefinedAs \mu_1 = \mu_2$, the matrix $M^t$ is determined by
	\beq
		M^t 
		\, = \,
		\tbt
		{(1-t) \, \mu^t}
		{t\,\mu^{t-1}}
		{-t\,\mu^{t+1}}
		{(t+1) \, \mu^t}.
		\label{eq.Mt}
	\eeq
\end{mylem}

Lemma~\ref{lem.powers} with $M = A_i$  determines explicit expressions for $A_i^t$. These expressions allow us to establish a bound on the norm of the response for each decoupled subsystem~\eqref{eq.statespace}. In Lemma~\ref{lem.Mblocks}, we provide a tight upper bound on $\| C_i A_i^t \|_2^2$ for each $t$ in terms of the spectral radius of the matrix $A_i$.  

	\vsp

\begin{mylem}\label{lem.Mblocks}
	The matrix $M$ in Lemma~\ref{lem.powers} satisfies
	\begin{align}\label{eq.specBound}
	\norm{\obt{1}{0}M^t}_2^2
	\, \le \,
	(t-1)^2 \rho^{2t} \, + \, t^2 \rho^{2t-2}
	\end{align}
where $\rho$ is the spectral radius of $M$. Moreover,~\eqref{eq.specBound} becomes equality if $M$ has repeated eigenvalues. 
\end{mylem}

	\vsp
\begin{myrem}
For Nesterov's accelerated algorithm with the parameters that optimize the convergence rate (cf.\ Table~\ref{tab:rates}), the matrix $\hat{A}_r$, which corresponds to the smallest non-zero eigenvalue of $Q$, $\lambda_r = m$, has an eigenvalue $1 - 2/\sqrt{3 \kappa + 1}$ with algebraic multiplicity two and incomplete sets of eigenvectors. Similarly, for both $\lambda_1 = L$ and $\lambda_r = m$, $\hat{A}_1$ and $\hat{A}_r$ for the heavy-ball method with the parameters provided in Table~\ref{tab:rates} have repeated eigenvalues which are, respectively, given by $(1 - \sqrt{\kappa})/(1 + \sqrt{\kappa})$ and $-(1 - \sqrt{\kappa})/(1 + \sqrt{\kappa})$.  
\end{myrem}		
	\vsp

	We next use Lemma~\ref{lem.Mblocks} with $M=A_i$ to establish an analytical expression for $J(t)$.
	
	\vsp
	
\begin{mythm}\label{thm.sigmax}
	For accelerated algorithms applied to convex quadratic problems, $J(t)$ in~\eqref{eq.tr} satisfies
	\beq
	J^2 (t)
	\, \le \, 
	\max 
	\left\{ (t-1)^2 \rho^{2t} \, + \, t^2 \rho^{2(t-1)},  
	\,
	\beta^{2t}/(1 \, + \, \beta^2)
	\right\}
	\non
	\eeq
where $\rho \DefinedAs \max_{i \, \leq \, r} \rho (A_i)$. Moreover, for the parameters provided in Table~\ref{tab:rates} 
	\beq
	J^2 (t)
	\, = \, 
	(t-1)^2 \rho^{2t} \, + \, t^2 \rho^{2(t-1)} .
	\label{eq.main}
	\eeq	
\end{mythm}
	
Theorem~\ref{thm.sigmax} highlights the source of disparity between the long and short term behavior of the response. While the geometric decay of $\rho^t$ drives $x^t$ to $x^{\star}$ as $t\rightarrow\infty$,  early stages are dominated by the algebraic term which induces a transient growth. We next provide tight bounds on the time $t_{\max}$ at which the largest transient response takes place and the corresponding peak value $J(t_{\max})$. Even though we derive the explicit expressions for these two quantities, our tight upper and lower bounds are more informative and easier to interpret. 

	\begin{center}	
	\begin{figure}[t]
		\centering
		\begin{tabular}{@{\hspace{0 cm}}r@{\hspace{-0.4 cm}}c@{\hspace{-0.45 cm}}c}
			\begin{tabular}{c}
				\rotatebox{90}{$\norm{x^t}_2^2$}
			\end{tabular}
			&
			\begin{tabular}{c}
				\includegraphics[width=.21\textwidth]{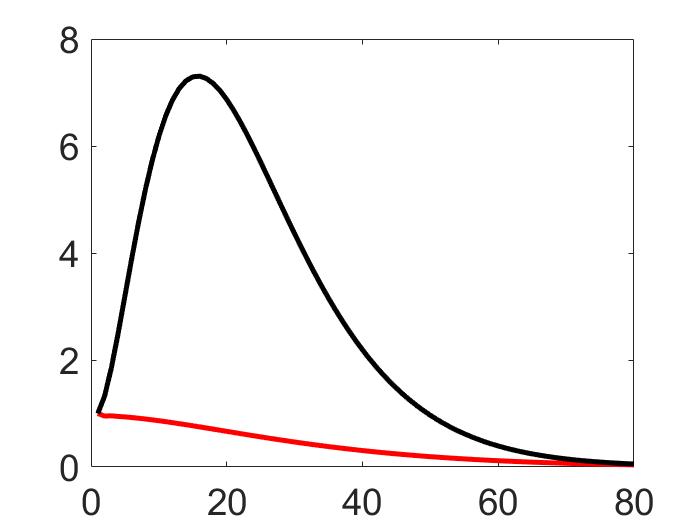}
				\\[-0.0 cm]
				{iteration number $t$}
			\end{tabular}
			&
			\begin{tabular}{c}
				\includegraphics[width=.21\textwidth]{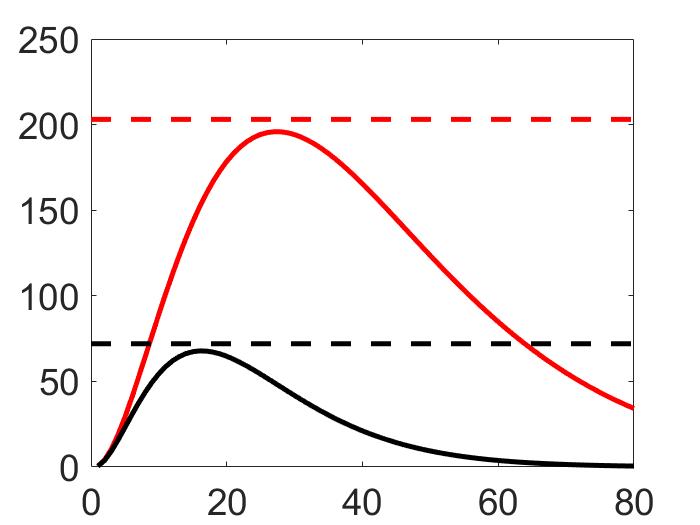}
				\\[-0.0cm]
				{iteration number $t$}
			\end{tabular}
			\\[-.05cm]
			&
			\begin{tabular}{cl}
				\begin{tabular}{c}
					\subfloat[\label{fig2.a}]{}	
				\end{tabular} & \hspace{-0.6 cm}
				\begin{tabular}{c}\vspace{-0.43 cm}
					$x^1 = x^0$
				\end{tabular}
			\end{tabular}
			&
			\begin{tabular}{cl}
				\begin{tabular}{c}
					\subfloat[\label{fig2.b}]{}	
				\end{tabular} & \hspace{-0.6 cm}
				\begin{tabular}{c}\vspace{-0.43 cm}
					$x^1=-x^0$
				\end{tabular}
			\end{tabular}
		\end{tabular}
		\caption{Dependence of the error in the optimization variable on the iteration number for the heavy-ball (black) and Nesterov's methods (red), as well as the peak magnitudes (dashed lines) obtained in Proposition~\ref{prop.kappa} for two different initial conditions with $\norm{x^1}_2 = \norm{x^0}_2 = 1$.}
		\label{fig.fig2}
	\end{figure}
\end{center}

	\vspace*{-2ex}
\begin{mythm}\label{thm.maxk}
For accelerated algorithms with the parameters provided in Table~\ref{tab:rates}, let $\rho \in [1/\mre,1)$. Then the rise time $t_{\max} \DefinedAs \argmax_t J(t)$ and the peak value $J(t_{\max})$ satisfy
\be
\label{eq.pposition}
\ba{c}
	- 
	{1}/{\log (\rho)}
	\; \le \;
	\ds t_{\max}
	\; \le \;
	1
	\, - \,
	{1}/{\log (\rho)}
	\\[0.3cm]
	-\dfrac{\sqrt{2} \rho}{\mre \log (\rho)}
	\; \le \;
	J(t_{\max})
	\; \le \;
	-\dfrac{\sqrt{2}}{\mre \, \rho \log (\rho)}.
	\ea
\non
\ee
\end{mythm}

For accelerated algorithms with the parameters provided in Table~\ref{tab:rates}, Theorem~\ref{thm.maxk} can be used to determine the rise time to the peak in terms of condition number $\kappa$. We next establish that both $t_{\max}$ and $J(t_{\max})$ scale as $\sqrt{\kappa}$. 

\vsp

\begin{myprop}
	\label{prop.kappa}
For accelerated algorithms with the parameters provided in Table~\ref{tab:rates}, the rise time $t_{\max} \DefinedAs \argmax_t J(t)$ and the peak value $J(t_{\max})$ satisfy
	\bi
	\item[(i)] Polyak's heavy-ball method with $\kappa\ge4.69$
\be
\ba{c}
	(\sqrt{\kappa} - 1)/2
	\; \le \;
	\ds t_{\max}
	\; \le \;
	(\sqrt{\kappa} + 3)/2
	\\[0.15cm]
	\dfrac{(\sqrt{\kappa} - 1)^2}{\sqrt{2} \, \mre(\sqrt{\kappa} + 1)}
	\; \le \;
	J(t_{\max})
	\; \le \;
	\dfrac{(\sqrt{\kappa} + 1)^2}{\sqrt{2} \, \mre(\sqrt{\kappa} - 1)}
	\ea
\non
\ee
	\item[(ii)] Nesterov's accelerated method with $\kappa\ge3.01$
\be
\ba{c}
	(\sqrt{3 \kappa + 1} - 2)/2
	\; \le \;
	\ds t_{\max}
	\; \le \;
	(\sqrt{3 \kappa + 1}  + 2)/2	
	\\[0.15cm]
	\dfrac{(\sqrt{3 \kappa + 1} - 2)^2}{\sqrt{2} \, \mre \sqrt{3 \kappa + 1}}
	\; \le \;
	J(t_{\max})
	\; \le \;
	\dfrac{3\kappa + 1}{\sqrt{2} \, \mre(\sqrt{3 \kappa + 1} - 2)}.
	\ea
\non
\ee
	\ei
\end{myprop}

	\vsp
	
In Proposition~\ref{prop.kappa}, the lower-bounds on $\kappa$ are only required to ensure that the convergence rate $\rho$ satisfies $\rho\ge1/\mre$, which allows us to apply Theorem~\ref{thm.maxk}. We also note that the upper and lower bounds on $t_{\max}$ and $J(t_{\max})$ are {\it tight\/} in the sense that their ratio converges to $1$ as $\kappa \rightarrow \infty$.
	
	\vspace*{-2ex}
\subsection{The role of initial conditions}
	\label{sec.initialConditions}

The accelerated algorithms need to be initialized with $x^0$ and $x^1 \in \R^n$. This provides a degree of freedom that can be used to potentially improve their transient performance. To provide insight, let us consider the quadratic problem with $Q=\diag \, (\kappa, 1)$. Figure~\ref{fig.fig2} shows the error in the optimization variable for Polyak's and Nesterov's algorithms as well as the peak magnitudes obtained in Proposition~\ref{prop.kappa} for two different types of initial conditions with $x^1=x^0$ and $x^1=-x^0$, respectively. For $x^1=-x^0$, both algorithms recover their worst-case transient responses. However, for $x^1=x^0$, Nesterov's method shows no transient growth.

Our analysis shows that large transient responses arise from the existence of non-normal modes in the matrices $A_i$. However, such modes do not move the entries of the state transition matrix $A_i^t$ in arbitrary directions. For example, using Lemma~\ref{lem.powers}, it is easy to verify that $A_r$ in~\eqref{eq.Ai}, associated with the smallest non-zero eigenvalue $\lambda_r = m$ of $Q$ in Nesterov's algorithm with the parameters provided by Table~\ref{tab:rates} has the repeated eigenvalue $\mu = 1-2/\sqrt{3\kappa+1}$ and $A_r^t$ is determined by~\eqref{eq.Mt} with $M = A_r$. Even though each entry of $A_r^t$ experiences a transient growth, its row sum is determined~by
\[
	A_r^t 
	\tbo{1}{1}
	= 
	\tbo{1 \, + \, 2t/(\sqrt{3\kappa+1} - 2)}{1 \, + \, 2t/\sqrt{3\kappa+1}}
	(1 \, - \, 2/\sqrt{3\kappa+1})^t
	\]
and entries of this vector are monotonically decaying functions of $t$. Furthermore, for $i < r$, it can be shown that the entries of $A_i^t \, [ \, 1 ~\, 1 \, ]^T$ remain smaller than $1$ for all $i$ and $t$. In Theorem~\ref{thm.noTransient}, we provide a bound on the transient response of Nesterov's method for {\em balanced\/} initial conditions with~$x^1=x^0$.

	\vsp
	
\begin{mythm}\label{thm.noTransient}
	For convex quadratic optimization problems, the iterates of Nesterov's accelerated method with a balanced initial condition $x^1 = x^0$ and parameters provided in Table~\ref{tab:rates} satisfy
	$
	\norm{x^t - x^\star}_2  \, \le \, \norm{x^0  - x^\star}_2.
	$
\end{mythm}

	\vsp
	
\begin{IEEEproof}
	See Appendix~\ref{app.ProofNoTransient}.
\end{IEEEproof}

	\vsp
	
It is worth mentioning that the transient growth of the heavy-ball method cannot be eliminated with the use of balanced initial conditions. To see this, we note that the matrices $A_r^t$ and $A_1^t$ for the heavy-ball method with parameters provided in Table~\ref{tab:rates} also take the form in~\eqref{eq.Mt} with $\mu = (1 - \sqrt{\kappa})/(1 + \sqrt{\kappa})$ and $\mu = -(1 - \sqrt{\kappa})/(1 + \sqrt{\kappa})$, respectively. In contrast to $A_r^t\obt{1}{1}^T$, which decays monotonically, 
	\[
	A_1^t
	\tbo{1}{1}
	\, = \,
	\tbo{1 \, + \, 2t \sqrt{\kappa}/(1-\sqrt{\kappa})}
	{1 \, + \,  2t \sqrt{\kappa}/(1+\sqrt{\kappa})}
	\dfrac{(1-\sqrt{\kappa})^t}{(1+\sqrt{\kappa})^t} 
	\]
experiences transient growth. It was recently shown that an averaged version of the heavy-ball method experiences smaller peak deviation than the heavy-ball method~\cite{danmal21}. We also note that adaptive restarting provides effective heuristics for reducing oscillatory behavior of accelerated algorithms~\cite{doncan15}.

	 \vsp

	\begin{myrem}
	For accelerated algorithms with the parameters provided in Table~\ref{tab:rates}, the initial condition that leads to the largest transient growth at any time $\tau$ is determined by 
	\[
	\hat{\psi}_r^0
	\, = \,
	c
	\obt
	{(1- \tau) \, \rho^\tau}
	{\tau \rho^{\tau-1}}^T,
	~
	\hat{\psi}_i^0 \, = \, 0
	~
	\mbox{for}
	~ 
	i \, \neq \, r
	\]
where $c \neq 0$ and $\hat{\psi}_r^0$ is the principal right singular vector of $C_r A_r^\tau$. Thus, the largest peak $J (t_{\max})$ occurs for $\{ \hat{\psi}_i^0 = 0$, $i \neq r \}$ and
	$
	\hat{\psi}_r^0
	=
	c
	\obt
	{(1- t_{\max})\, \rho^{t_{\max}}}
	{t_{\max} \,\rho^{t_{\max}-1}}^T,
	$
	where tight bounds on $t_{\max}$  are established in Proposition~\ref{prop.kappa}. 	
\end{myrem}

	\vsp

	\begin{myrem}
For $\lambda_i=0$ in~\eqref{eq.statespace}, $|\hat{x}_i^t-\hat{x}_i^\star|$ decays monotonically with a linear rate $\beta$ and only non-zero eigenvalues of $Q$ contribute to the transient growth. Furthermore, for the parameters provided in Table~\ref{tab:rates}, our analysis shows that $J^2 (t) = \max_{i \, \le \, r} \norm{C_i A_i^t}_2^2$. In what follows, we provide bounds on the largest deviation from the optimal solution for Nesterov's algorithm for general strongly convex problems.
	\end{myrem}

	\vspace*{-1ex}
 \section{General strongly convex problems}
 	\label{sec.main}

In this section, we combine a Lyapunov-based approach with the theory of IQCs to provide bounds on the transient growth of Nesterov's accelerated algorithm for the class $\mathcal{F}_{\mf}^{\Lf}$ of $m$-strongly convex and $L$-smooth functions. When $f$ is not quadratic, first-order algorithms are no longer LTI systems and eigenvalue decomposition cannot be utilized to simplify analysis. Instead, to handle nonlinearity and obtain upper bounds on $J$ in~\eqref{eq.tr}, we augment standard quadratic Lyapunov functions with the objective error. 

	For $f \in \mathcal{F}_{\mf}^{\Lf}$, algorithm~\eqref{eq.1st-order} is invariant under translation. Thus, without loss of generality, we assume that $x^\star = 0$ is the unique minimizer of~\eqref{eq.optProb} with $f(0)=0$. In what follows, we present a framework based on Linear Matrix Inequalities (LMIs) that allows us to obtain time-independent bounds on the error in the optimization variable. This framework combines certain IQCs~\cite{megran97} with Lyapunov \mbox{functions of the form}
\begin{equation}\label{eq.LyapForm}
	V(\psi) 
	\; = \;
	\psi^T X \psi 
	\; + \;
	\theta f(C \psi)
\end{equation}
which consist of the objective function evaluated at $C \psi$ and a quadratic function of $\psi$, where $X$ is a positive definite matrix. 

The theory of IQCs provides a convex control-theoretic approach to analyzing optimization algorithms~\cite{lesrecpac16} and it was recently employed to study convergence and robustness of the first-order methods~\cite{hules17,fazribmor18,cyrhuvanles18,dhikhojovTAC19,mohrazjovTAC21,mogjovAUT21}. The type of Lyapunov functions in~\eqref{eq.LyapForm} was introduced in~\cite{polshc17,fazribmor18} to study convergence for convex problems. For Nesterov's accelerated algorithm, we demonstrate that this approach provides {\em orderwise-tight\/} analytical upper bounds on $J(t)$.

Nesterov's accelerated algorithm can be viewed as a feedback interconnection of linear and nonlinear components 
\begin{subequations}
	\label{eq.NesInSS}
	\beq
	\ba{rcl}
	\psi^{t+1}
	& \!\!\! = \!\!\! &
	A \, \psi^t
	\; + \;
	B\, u^t
	\\[0.1cm]
	y^t
	& \!\!\! = \!\!\! &
	C_y \,
	\psi^t,
	\quad
	u^{t}\;=\;
	\Delta ( y^t )
	\ea
	\label{eq.ss-ns}
	\eeq
where the LTI part of the system is determined by
	\be
	\ba{rcl}
	A
	& \!\!\! = \!\!\! &
	\tbt{0}{I}{-\beta I}{(1+\beta) I},~~
	B
	\, = \,
	\tbo{0}{-\alpha I},~~
	C_y
	\, = \,
	\obt{-\beta I}{(1 + \beta)I} 
	\label{eq.ss-na}
	\ea
	\ee 
and the nonlinear mapping $\Delta$$:\R^n\rightarrow\R^n$ is 
	$
	\Delta (y)
	\DefinedAs
	\nabla f(y).
	$ 
Moreover, the state vector $\psi^t$ and the input $y^t$ to $\Delta$ are determined by 
	\beq
	\psi^t 
	\, \DefinedAs \, 
	\left[ 
	\ba{c} {x^t} \\[-0.cm]{x^{t+1}} \ea
	\right],
	~~
	y^t
	\, \DefinedAs \, 
	(1 + \beta)  x^{t+1}
	\, - \, 
	\beta x^t.
	\label{eq.state-output-Nesterov}
	\eeq 
\end{subequations}
For smooth and strongly convex functions $f\in\mathcal{F}_\mf^\Lf$, $\Delta$ satisfies the quadratic inequality~\cite[Lemma 6]{lesrecpac16}
	\begin{subequations}
	\label{eq.PI-iqc}
	\be
	\tbo{y \,-\, y_0}{\Delta (y) \,-\, \Delta(y_0)}^T
	\!\Pi
	\tbo{y \,-\, y_0}{\Delta (y) \,-\, \Delta(y_0)}
	\;\geq\; 
	0
	\label{eq.IQC}
	\eeq
	for all $y$, $y_0 \in \bbR^n$, where the matrix $\Pi$ is given by
	\begin{align}\label{eq.PI}
		\Pi 
		\;\DefinedAs\;
		\tbt{-2 \mf \Lf I}{(\Lf + \mf)I }{(\Lf + \mf)I}{-2I}.
	\end{align}
Using $u^t \DefinedAs \Delta (y^t)$ and $y^t \DefinedAs C_y \psi^t$ and evaluating~\eqref{eq.IQC} at $y = y^t$ and $y_0 = 0$ leads to,
	\be
	\tbo{\psi^t}{u^t}^T
	\!
	M_1
	\tbo{\psi^t}{u^t}
	\;\geq\; 
	0
	\label{eq.IQC1}
	\eeq
where
	\beq
	\ba{rrl}
	M_1 
	& \!\!\! \DefinedAs \!\!\! &
	\tbt{C_y^T}{0}{0}{I}\Pi \tbt{C_y}{0}{0}{I}
	\, = \,
	\tbt{-2 \mf \Lf C_y^T C_y}{(\Lf + \mf) C_y^T }{(\Lf + \mf) C_y}{-2I}.
	\ea
	\label{eq.M1}
	\eeq		
	\end{subequations}
 
In Lemma~\ref{lem.Mahyar}, we provide an upper bound on the difference between the objective function at two consecutive iterations of Nesterov's algorithm. In combination with~\eqref{eq.PI-iqc}, this result allows us to utilize Lyapunov function of the form~\eqref{eq.LyapForm} to establish an upper bound on transient growth. We note that variations of this lemma have been presented in~\cite[Lemma~5.2]{fazribmor18} and in~\cite[Lemma~3]{mohrazjovTAC21}.
	
	\vsp
	
\begin{mylem}
	\label{lem.Mahyar}
	Along the solution of Nesterov's accelerated algorithm~\eqref{eq.NesInSS}, the function $f\in\mathcal{F}_{\mf}^{\Lf}$ with $\kappa\DefinedAs\Lf/\mf$ satisfies
	\begin{subequations}
	\label{eq.f-lemma}
	\beq
		f(x^{t+2}) \, - \, f(x^{t+1}) 
		\; \le \;
		\dfrac{1}{2} 
		\tbo{\psi^t}{u^t}^T
		M_2
		\tbo{\psi^t}{u^t}
		\label{eq.f-lemma1}
\eeq
where the matrix $M_2$ is given by
	\beq
	\ba{rcl}
	M_2
	&\!\!\! \DefinedAs \!\!\! &
	\tbt{- \mf C_2^T C_2}{C_2^T}{C_2}{-\alpha(2 - \alpha L)I},
	~~
	C_2
	\, \DefinedAs \,
	\obt{-\beta I}{\beta I}.
	\ea
	\label{eq.M2}
	\eeq
	\end{subequations}		
	\end{mylem}

Using Lemma~\ref{lem.Mahyar}, we next demonstrate how a Lyapunov function of the form~\eqref{eq.LyapForm} with $\theta \DefinedAs 2 \theta_2$ and $C \DefinedAs [\, 0 ~\, I \, ]$ in conjunction with property~\eqref{eq.PI-iqc} of the nonlinear mapping $ \Delta $  can be utilized to obtain an upper bound on $\norm{x^t}_2^2$.

	\vsp
	
\begin{mylem}
	\label{lem.generalNester}
	Let $M_1$ be given by~\eqref{eq.M1} and let $M_2$ be defined in Lemma~\ref{lem.Mahyar}. Then, for any positive semi-definite matrix $ X $ and nonnegative scalars $ \theta_1$ and $\theta_2$ that satisfy
	\beq
	W 
	\, \DefinedAs \,
	\tbt{A^T X\, A - X }{A^T X\, B}{B^T\, X\, A}{B^T\, X\, B}
	\, + \,
	\theta_1 
	M_1
	\, + \,
	\theta_2
	M_2
	\, \preceq \,
	0
	\label{eq.LMI2}
	\eeq
the transient growth of Nesterov's accelerated algorithm~\eqref{eq.NesInSS} for all $t \ge 1$ is upper bounded by 
	\begin{align}\label{eq.RHSup}
	\norm{x^t}_2^2
	\, \leq \,
	\dfrac{ \lambda_{\max}(X)\norm{x^0}_2^2 
	\, + \, 
	(\lambda_{\max}(X) + L\theta_2) \norm{x^{1}}_2^2 }{\lambda_{\min}(X) + m\theta_2}.
	\end{align}
\end{mylem} 
In Lemma~\ref{lem.generalNester}, the Lyapunov function candidate $V(\psi)\DefinedAs\psi^TX\psi + 2\theta_2 f ([\, 0~I \,] \psi)$ is used to show that the state vector $\psi^t$ is confined within the sublevel set 
	$
	\{\psi \in \R^{2n} \, | \, V(\psi) \le V(\psi^0)\}
	$
associated with $V(\psi^0)$. We next  establish an {\em order-wise\/} tight upper bound on $\norm{x^t}_2$ that scales linearly with $\sqrt{\kappa}$ by finding a feasible point to LMI~\eqref{eq.LMI2} in Lemma~\ref{lem.generalNester}.
	
	\vsp
		
\begin{mythm}\label{thm.main}
	For $f \in \mathcal{F}_{\mf}^{\Lf}$ with the condition number $\kappa \DefinedAs L/m$, the iterates of Nesterov's accelerated algorithm~\eqref{eq.NesInSS} for any stabilizing parameters $\alpha \leq 1/L$ and $\beta < 1$ satisfy
	\begin{subequations}
		\label{eq.thm.main}
	\be
	\label{eq.hesUp}
	\norm{x^t}_2^2
	\, \le \,
	\kappa
	\left(
	\dfrac{1 + \beta^2}{\alpha \beta L} 
	\, 
	\norm{x^0}_2^2 
	\, + \,
	(1 + \dfrac{1 + \beta^2}{\alpha \beta L} )
	\, 
	\norm{x^1}_2^2
	\right).
	\ee
	Furthermore, for the conventional values of parameters
	\beq
	\label{eq.convParam}
	\alpha \, = \, 1/L,
	~
	\beta \, = \, (\sqrt{\kappa}-1)/(\sqrt{\kappa}+1)
	\eeq	
	the largest transient error, defined in~\eqref{eq.tr}, satisfies
	\beq	
	\dfrac{\sqrt{2} \, (\sqrt{\kappa} - 1)^2}{\mre \sqrt{\kappa}}
	\, \le 
	\sup_{\{ t\,\in\,\mathbb{N}, \, f\,\in\,\mathcal{F}_m^L \}} J (t)
	\, \le \,
	\sqrt{3 \kappa + \dfrac{4 \kappa}{\kappa - 1}}.
	\label{eq.upperLowerGeneralSCVX}
	\eeq
\end{subequations}
\end{mythm}
	
	For balanced initial conditions, i.e., $x^1 = x^0$, Nesterov established the upper bound $\sqrt{\kappa +1}$ on $J$ in~\cite{nes18book}. Theorem~\ref{thm.main} shows that similar trends hold without restriction on initial conditions. Linear scaling of the upper and lower bounds with $\sqrt{\kappa}$ illustrates a potential drawback of using Nesterov's accelerated algorithm in applications with limited time budgets. As $\kappa \rightarrow \infty$, the ratio of these bounds converges to $\mre \sqrt{3/2} \approx 3.33$, thereby demonstrating that the largest transient response for all $f \in \mathcal{F}_m^L$ is within the factor of $3.33$ relative to the bounds established in Theorem~\ref{thm.main}. 

	\newpage
\section{Concluding remarks}
\label{sec.conc}

We have examined the impact of acceleration on transient responses of first-order optimization algorithms. Without imposing restrictions on initial conditions, we establish bounds on the largest value of the Euclidean distance between the optimization variable and the global minimizer. For convex quadratic problems, we utilize the tools from linear systems theory to fully capture transient responses and for general strongly convex problems, we employ the theory of integral quadratic constraints to establish an upper bound on transient growth. This upper bound is proportional to the square root of the condition number and we identify quadratic problem instances for which accelerated algorithms generate transient responses which are within a constant factor of this upper bound. Future directions include extending our analysis to nonsmooth optimization problems and devising algorithms that balance acceleration with quality of transient responses. 

	\vspace*{-1ex}
\appendix

\subsection{Proofs of Section~\ref{sec.trans}}
\label{app.lemmaProofs}
We first present a technical lemma that we use in our proofs.

	\vsp
	
\begin{mylem}\label{lem.a(t)}
	For any  $\rho\in[1/\mre,1)$,  $a(t)\DefinedAs t\rho^t$ satisfies
	\[
	\argmax_{t\,\ge\, 1} \, a(t) 
	\, = \, 
	{-1}/{\log(\rho)},
	~
	\max_{t\,\ge\,1} \, a(t)
	\, = \, 
 	{-1}/{(\mre\log(\rho))}.
	\]
\end{mylem}
\begin{IEEEproof}
	Follows from the fact that  $\mrd a/\mrd t= \rho^t(1+t\log( \rho))$ vanishes at $t=-1/\log(\rho)$.
\end{IEEEproof}

\subsubsection{Proof of Lemma~\ref{lem.powers}}
For $\mu_1\neq\mu_2$, the eigenvalue decomposition of $M$ is determined by
\begin{align*}
M
\, = \,
\dfrac{1}{\mu_2 - \mu_1}
\tbt
{1}
{1}
{\mu_1}
{\mu_2}
\tbt
{\mu_1}
{0}
{0}
{\mu_2}
\tbt
{\phantom{-}\mu_2}
{-1}
{-\mu_1}
{\phantom{-}1}.
\end{align*}
Computing the $t$th power of the diagonal matrix and multiplying throughout completes the proof for $\mu_1\neq\mu_2$. For $\mu_1 = \mu_2 \AsDefined \mu$, $M$ admits the Jordan canonical form	
\begin{align*}
M
\;=\;
\tbt
{1}
{0}
{\mu}
{1}
\tbt
{\mu}
{1}
{0}
{\mu}
\tbt
{\phantom{-}1}
{0}
{-\mu}
{1}
\end{align*}
and the proof follows from
\begin{align*}
\tbt
{\mu}
{1}
{0}
{\mu}^t
	= \,	
\tbt
{\mu^t}
{t \, \mu^{t-1}}
{0}
{\mu^t}.
\end{align*}

\subsubsection{Proof of Lemma~\ref{lem.Mblocks}}
From Lemma~\ref{lem.powers}, it follows
\begin{align*}
\obt{1}{0}M^t
\,=\,
\obt{-\ds{\sum_{i \, = \, 0}^{t-2}} \; \mu_1^{i+1}\mu_2^{t-1-i}}{\ds{\sum_{i \, = \, 0}^{t-1}} \; \mu_1^i\mu_2^{t-1-i}},
\end{align*}
where $\mu_1$ and $\mu_2$ are the eigenvalues of $M$. Moreover,
\begin{align*}
|\sum_{i \, = \, 0}^{t-2} \mu_1^{i+1} \mu_2^{t-1-i}|
& \le
\sum_{i \, = \, 0}^{t-2} |\mu_1^{i+1} \mu_2^{t-1-i}|
\le
\sum_{i \, = \, 0}^{t-2} \rho^t
\le\!
(t-1)\rho^t
\\[0.1cm]
|\sum_{i \, = \, 0}^{t-1} \mu_1^{i} \mu_2^{t-1-i}|
&\le
\sum_{i \, = \, 0}^{t-1} |\mu_1^{i} \mu_2^{t-1-i}|
\le
\sum_{i \, = \, 0}^{t-1}\rho^{t-1}
\le
t\rho^{t-1}
\end{align*}
by triangle inequality. Finally, for  $\mu_1=\mu_2\in\R$, we have $\rho =  |\mu_1|=|\mu_2|$ and these inequalities become equalities. 

\subsubsection{Proof of Theorem~\ref{thm.sigmax}}
	 Let $\mu_{1i}$ and $\mu_{2i}$ be the eigenvalues and let
	$
	\rho_i 
	=
	\max \, \{|\mu_{1i}|,|\mu_{2i}|\}
	$
be the spectral radius of $A_i$. We can use Lemma~\ref{lem.Mblocks} with $M \DefinedAs A_i$ to obtain
	\beq
	\ba{rcl}
	\displaystyle{\max_{i \, \le \, r}}  
	\;
	\norm{ C_iA_i^t }_2^2
	& \!\!\! \le \!\!\! &
	\displaystyle{\max_{i \, \le \, r}}  
	\;
	\left( (t-1)^2 \rho_i^{2t} \, + \, t^2 \rho_i^{2t-2} \right)
 	\\[0.15cm]
	 & \!\!\! \le \!\!\! &
	(t-1)^2 \rho^{2t} \, + \, t^2 \rho^{2t-2}
	\ea
	\label{eq.leftIneq}
	\eeq
where $\rho\DefinedAs\max_{i\le r} \rho_i$. For the parameters provided in Table~\ref{tab:rates}, the matrices $A_1$ and $A_r$, that correspond to the largest and smallest non-zero eigenvalues of $Q$, i.e.,  $\lambda_1 = L$ and $\lambda_r = m$, respectively,  have the largest spectral radius~\cite[Eq.~(64)]{mohrazjovTAC21},
	\begin{align}
	\label{eq.maxrho}
	\rho
	\, = \,
	\rho_1 
	\, = \,
	\rho_r
	\, \ge \,
	\rho_i,
	~
	i \, = \, 2,\ldots,r-1
	\end{align}
and $A_r$ has repeated eigenvalues. Thus, we can write 	
	\beq
	\ba{l}
	\displaystyle{\max_{i \, \le \, r}} 
	\;
	\norm{ C_iA_i^t }_2^2
	\, \ge \,
 	\norm{ \obt{1}{0} \! A_r^t}_2^2
 	~=
 	\\[0.1cm]
	(t-1)^2 \rho_r^{2t} \, + \, t^2 \rho_r^{2t-2}
	\, = \,
	(t-1)^2 \rho^{2t} \, + \, t^2 \rho^{2t-2}
	\label{eq.rightIneq}
	\ea
	\eeq
where the first equality follows from Lemma~\ref{lem.Mblocks} applied to $M \DefinedAs A_r$ and the second equality follows from~\eqref{eq.maxrho}. Finally, combining~\eqref{eq.leftIneq} and~\eqref{eq.rightIneq} with $\beta<\rho$ and Proposition~\ref{prop.Jquad} completes the proof.

\subsubsection{Proof of Theorem~\ref{thm.maxk}}
Let $a(t) \DefinedAs t \rho^t$.  Theorem~\ref{thm.sigmax} implies
$
J^2(t)
=
\rho^2a^2(t-1)
+ 
\rho^{-2}a^2(t)
$ 
and, for $t\ge1$, $J(t)$ has only one critical point, which is a maximizer. Moreover, since $\mrd J^2(t)/{\mrd t}$ is positive at $t=-1/\log (\rho)$ and negative at $t=1-1/\log (\rho)$, we conclude that the maximizer lies between $-1/\log (\rho)$ and  $1-1/\log( \rho)$.
Regarding $\max_t J(t)$, we note that $
\sqrt{2}\rho a(t-1)
\le
J(t)
\le
\sqrt{2}a(t)/\rho
$
and the proof follows from $\max_{t \ge 1} a(t)= - 1/({\mre\log( \rho)})$ (cf.\ Lemma~\ref{lem.a(t)}). 

\subsubsection{Proof of Proposition~\ref{prop.kappa}}

Since for all $a \le 1$, we have~\cite{top06}
	\[
a \, \le \, -\log \, (1-a) \, \le \, {a}/{(1-a)}
	\]
$
\rho_\hb 
=  
1  -  {2}/({\sqrt{\kappa} + 1})
$
and 
$
\rho_\na 
= 
1  -  {2}/({\sqrt{3\kappa+1}})
$ 
satisfy
\begin{align*}
{2}/{(\sqrt{\kappa}+1)}&\;\le\;-\log(\rho_\hb) \;\le\;{2}/{(\sqrt{\kappa}-1)}
\\[0.15 cm]
{2}/{\sqrt{3\kappa+1}}&\;\le\;-\log(\rho_\na) \;\le\;{2}/{(\sqrt{3\kappa+1}-2)}.
\end{align*}
The conditions on $\kappa$ ensure that $\rho_\hb$ and $\rho_\na$ are not smaller than $1/\mre$ and we combine the above bounds with Theorem~\ref{thm.maxk} to complete the proof.
 
	\vspace*{-2ex}
\subsection{Proof of Theorem~\ref{thm.noTransient}}
\label{app.ProofNoTransient}

The condition $x_0=x_1$ is equivalent to $\hat{x}^0_i=\hat{x}^1_i$ in~\eqref{eq.statespace}. Thus, for $\lambda_i=0$, equation~\eqref{eq.zeroDynamics} yields
$
\hat{x}_i^t=\hat{x}_i^0=\hat{x}_i^\star.
$
For $\lambda_i\neq 0$, we have $\hat{\psi}_i^0-\hat{\psi}_i^\star = \obt{\hat{x}_i^0}{\hat{x}_i^0}^T$ and, hence,
	\begin{subequations}\label{eq.temp3+4}
	\beq
		\dfrac{\norm{x^t-x^\star}_2}{\norm{x^0-x^\star}_2} 
		\, \le \,
		\max_{i\,\le\, r}
		\dfrac{\abs{\hat{x}^t_i-\hat{x}_i^\star}}{\abs{\hat{x}^t_0-\hat{x}_i^\star}}
		\, = \,
		\max_{i\,\le\, r}~\abs{C_iA_i^t\tbo{1}{1}}
		\label{eq.temp3}
	\eeq	
where the equality follows from~\eqref{eq.xi-i}. 
To bound the right-hand side, we use  Lemma~\ref{lem.powers} with $M=A_i$ to obtain \begin{align}\label{eq.temp4}
	\omega_t(\mu_{1i},\mu_{2i})\,=\,\obt{1}{0}A_i^t\obt{1}{1}^T
\end{align}
\end{subequations}
  where $\mu_{1i}$ and $\mu_{2i}$ are the eigenvalues of $A_i$ and 
\begin{align}\label{eq.omega}
	\omega_t(z_1,z_2) 
	\, \DefinedAs \,
	\sum_{i \, = \, 0}^{t-1}z_1^iz_2^{t-1-i} \, - \, \sum_{i \, = \, 1}^{t-1}z_1^iz_2^{t-i}
\end{align}
for any $t\in \N$ and $z_1,z_2\in\mathbb{C}$. 

For Nesterov's accelerated method, the characteristic polynomial  
$
\det(zI-A_i) = z^2-(1+\beta)h_iz + \beta h_i
$ yields
$
\mu_{1i},\mu_{2i}= ((1+\beta)h_i \pm \sqrt{(1+\beta)^2h_i^2-4\beta h_i} )/2
$,
where $\lambda_i$ is the eigenvalue of $Q$ and $h_i\DefinedAs1-\alpha\lambda_i$. For the parameters provided in Table~\ref{tab:rates}, it is easy to show that:
\begin{itemize}
	\item For $\lambda_i\in[m,1/\alpha]$, we have $h_i\in[0,4\beta/(1+\beta)^2]$ and $\mu_{1i}$ and $\mu_{2i}$ are complex conjugates of each other and lie on a circle of radius $\beta/(1+\beta)$ centered at $z=\beta/(1+\beta)$.
	\item  For  $\lambda_i\in(1/\alpha, L]$, $\mu_{1i}$ and $\mu_{2i}$ are real with opposite signs and can be sorted to satisfy $\abs{\mu_{2i}}<\abs{\mu_{1i}}$ with $-1\le \mu_{1i}\le0\le\mu_{2i}\le1/3$.
\end{itemize} 

The next lemma provides a unit bound on $\abs{w_t(\mu_{1i},\mu_{2i})}$ for both of the above cases. 
 
 	\vsp
	
\begin{mylem}
	\label{lem.helperNoTransient}
For any $z=l\cos(\theta)\mre^{i\theta}\in\mathbb{C}$ with $|\theta|\le \pi/2$ and $0\le l\le1$, and for any real scalars $(z_1,z_2)$ such that $-1\le z_1\le0\le z_2\le1/3$,  and $z_2<-z_1$, the function $\omega_t$ in~\eqref{eq.omega} satisfies 
	$
	\abs{\omega_t(z,\bz)}
	\le 
	1
	$
and
	$
	\abs{\omega_t(z_1,z_2)}
	\le 
	1
	$
for all $t\,\in\,\N$, where $\bar{z}$ is the complex conjugate of $z$.
\end{mylem}

	\vsp
	
\begin{IEEEproof}
	Since $\omega_1(z_1,z_2)=1$, we assume $t\ge 2$.
	We first address $\theta=0$, i.e.,  $z=l \in \bbR$ and
	$\omega_t(z,\bz)= tl^{t-1} - (t-1)l^t$. We note that
	$
	{\mrd \omega_t}/{\mrd l}
	=
	t(t-1)(l^{t-2}-l^{t-1}) = 0
	$
only if $l\in\{0,1\}$. This in combination with $l\in[0,1]$ yield 
	$\abs{\omega_t(l,l)}\le\max\{\abs{\omega_t(1,1)},\abs{\omega_t(0,0)}\}\le1.$	
	
	To address $\theta\neq 0$, we note that $b(t)\DefinedAs \sin(t\theta)/t$ satisfies
	\begin{align}
	\label{eq.helperSinttheta}
		\abs{b(t)}
		\;\le\;\abs{\sin(\theta)}
	\end{align}
which follows from
	\begin{align*}
		\abs{\sin(t\theta)}
		&\;=\;
		\abs{\sin((t-1)\theta)\cos(\theta) \, + \, \cos((t-1)\theta)\sin(\theta)}
		\;\le\;
		\abs{\sin((t-1)\theta)} \, + \, \abs{\sin(\theta)}.
	\end{align*} 
	For $z=l\cos(\theta)\mre^{i\theta}$, we have
	\beq
		\omega_t(z,\bz) 
		\, = \,
	({z^t-\bz^t - z\bz(z^{t-1}-\bz^{t-1})} )/{(z-\bz)}
	\, = \,
		(l\cos(\theta))^{t-1} ({\sin(t\theta) - l\cos(\theta)\sin((t-1)\theta)} )/{\sin(\theta)}.
		\non
	\eeq
	Thus, ${\mrd \omega_t}/{\mrd l}=0$ only if $l=0$, $1$, or $l^\star\DefinedAs b(t)/(b(t-1)\cos(\theta))$. Moreover, it is easy to show that 
	\begin{align*}
		\omega_t(z,\bz)\;=\; \left\{
		\ba{lrcl}
		0,
		&l&\!\!\!=\!\!\!&0\\[0.1cm]
		( \cos(\theta) ) ^{t-1}\cos((t-1)\theta),
		&l&\!\!\!=\!\!\!&1\\[0.1cm]
		{(l^\star\cos(\theta))^{t-1}b(t)}/{ \sin(\theta)},
		&l &\!\!\!=\!\!\!&l^\star. 
		\ea
		\right.
	\end{align*}
	Combining this with~\eqref{eq.helperSinttheta} completes the proof for complex $z$. 
	
	To address the case of $z_1$, $z_2 \in \bbR$, we note that
	$
		\omega_t(z_1,z_2) 
		=
		{\left(z_1^t(1-z_2)-z_2^t(1-z_1) \right)}/{(z_1-z_2)}.
	$
	Thus, differentiating with respect to $z_1$ yields 
	\begin{align*}
		\dfrac{\mrd \omega_t}{\mrd z_1}
		\;=\;(1-z_2)
		\dfrac{(t-1)z_1^{t-1}-z_2\sum_{i \, = \, 0}^{t-2}z_1^{t-2-i}z_2^i}{z_1-z_2}.
	\end{align*}
Moreover, from  $\abs{z_2} < \abs{z_1}$, it follows that
\begin{align*}
	(t-1)\abs{z_1^{t-1}}
	\;>\;
	\abs{z_2}\sum_{i \, = \, 0}^{t-2}\abs{z_1^{t-2-i}z_2^i}
	\;>\;
	\abs{z_2\sum_{i \, = \, 0}^{t-2}z_1^{t-2-i}z_2^i}.
\end{align*}
Therefore, ${\mrd \omega_t}/{\mrd z_1}\neq 0$ over our range of interest for $z_1, z_2$. Thus, $\omega_t(z_1,z_2)$ may take its extremum only at the boundary $z_1\in\{0,-1\}$, i.e.
	$\abs{\omega_t(z_1,z_2)}\le
	\max\{
	\abs{\omega_t(0,z_2)},\abs{\omega_t(1,z_2)}\}.
	$
Finally, it is easy to show that $\abs{\omega_t(0,z_2)} =\abs{z_2^{t-1}}<1$, and 
$
	\abs{\omega_t(-1,z_2)}
	=
	{\abs{(-1)^t(z_2-1)+2z_2^t}}/{(1+z_2)}
	\le
	1.
$
\end{IEEEproof}

	\vsp
	
We complete the proof of Theorem~\ref{thm.noTransient} by noting that the eigenvalues of $A_i$ for Nesterov's algorithm with parameters provided in Table~\ref{tab:rates} satisfy the conditions in Lemma~\ref{lem.helperNoTransient}.

	\vspace*{-2ex}
\subsection{Proofs of Section~\ref{sec.main}} 
\label{app.proofLMI}
 
\subsubsection{Proof of Lemma~\ref{lem.Mahyar}}
For any $f\in\mathcal{F}_\mf^\Lf$, the $L$-Lipschitz continuity of the gradient $\nabla f$, 	\begin{subequations}
		\beq
			f(x^{t+2}) \, - \, f(y^t) 
			\, \le \,   
			( \nabla f (y^t) )^T (x^{t+2} -y^t) 
			\, + \, 
			\dfrac{L}{2} \, \norm{x^{t+2} -y^t}_2^2
			\label{ineqTemp1}
		\eeq		
and the $m$-strong convexity of $f$, 
		\beq
			f(y^{t}) \, - \, f(x^{t+1})
			\, \le \,
			( \nabla f (y^t) )^T ( y^t - x^{t+1} ) 
			\, - \, 
			\dfrac{m}{2} \, \norm{y^t - x^{t+1}}_2^2
			\label{ineqTemp2}
		\eeq
	\end{subequations}
can be used to show that~\eqref{eq.f-lemma} holds along the solution of Nesterov's accelerated algorithm~\eqref{eq.NesInSS}. In particular, for~\eqref{eq.NesInSS} we have $u^t \DefinedAs \nabla f (y^t)$ and 
	\beq
	\ba{rcl}
	x^{t+2} \, - \, y^t
	& \!\!\! = \!\!\! &
	-\alpha u^t
	\\[0.1cm]
	y^{t} \, - \, x^{t+1}
	& \!\!\! = \!\!\! &
	\beta (x^{t+1} \, - \, x^t)
	\, = \,
	\obt{-\beta I}{\beta I} \psi^t.
	\ea
	\label{eq.key-Nesterov}
	\eeq
Substituting~\eqref{eq.key-Nesterov} into~\eqref{ineqTemp1} and~\eqref{ineqTemp2} and adding the resulting inequalities completes the proof.

	\vspace*{1ex}
\subsubsection{Proof of Lemma~\ref{lem.generalNester}}
Pre- and post-multiplication of LMI~\eqref{eq.LMI2} by $(\eta^t)^T$ and $\eta^t \DefinedAs [ \, (\psi^t)^T \; (u^t)^T \, ]^T$ yields
\beq
\ba{rcl}
0
& \!\!\! \ge \!\!\! &
(\eta^t)^{T}
\tbt{A^T X\, A - X}{A^T X\, B}{B^T\, X\, A}{B^T\, X \,B}\eta^t
\,+ \, 
\theta_1 (\eta^t)^T M_1 \eta^t
\,+\,
\theta_2 (\eta^t)^T M_2 \eta^t
	\\[0.45cm]
	& \!\!\! \ge \!\!\! &
(\eta^t)^{T}
\tbt{A^T X\, A - X}{A^T X\, B}{B^T\, X\, A}{B^T\, X \,B}\eta^t
	\, + \, 
\theta_2 (\eta^t)^T M_2 \eta^t
	\ea
	\non
	\eeq
where the second inequality follows from~\eqref{eq.IQC1}. 
This yields
\begin{align}\label{ineqTemp5}
	0
	\, \le \,
	\hat{V}(\psi^t) 
	\, - \,
	\hat{V}(\psi^{t+1}) 
	\, - \,
	\theta_2 (\eta^t)^T M_2 \eta^t
\end{align}
where $ \hat{V}(\psi) \DefinedAs \psi^TX\psi $.
Also, since Lemma~\ref{lem.Mahyar} implies
\begin{equation}
\label{eq.lowerBoundOnetaTMeta}
	-(\eta^t)^T M_2  \eta^t 
	\, \le \,
	2\left(f(x^{t+1}) \,-\, f(x^{t+2})\right)
\end{equation}	
combining~\eqref{ineqTemp5} and~\eqref{eq.lowerBoundOnetaTMeta} yields
\begin{align*}
\hat{V}(\psi^{t+1}) \,+\, 2 \theta_2 f(x^{t+2}) 
\;\leq\;
\hat{V}(\psi^t) \,+\, 2 \theta_2 f(x^{t+1}).
\end{align*}
Thus, using induction, we obtain the uniform upper bound
\begin{align}\label{eq.levset}
\hat{V}(\psi^t) \,+\, 2 \theta_2 f(x^{t+1}) 
\;\leq\;
\hat{V} (\psi^0) \,+\, 2 \theta_2 f(x^{1}).
\end{align}
This allows us to bound  $\hat{V}$ by writing
\begin{subequations}\label{eq.uploLyap}
	\begin{align}
	\lambda_{\min}(X)\norm{\psi}_2^2 
	\;\leq\;
	\hat{V}(\psi)
	\;\leq\;
	\lambda_{\max}(X) \norm{\psi}_2^2.
	\end{align}
	We can also upper and lower bound $f \in \mathcal{F}_m^L$ as
	\begin{align}
	m\norm{x}_2^2
	\;\leq\;
	2 f(x) 
	\;\leq\;
	L \norm{x}_2^2.
	\end{align}
\end{subequations}
Finally, combining~\eqref{eq.levset} and~\eqref{eq.uploLyap} yields 
\beq
	\ba{l}
	\lambda_{\min}(X)\norm{\psi^{t}}_2^2 \, + \, m\, \theta_2  \norm{x^{t+1}}_2^2 
	\, \le \,
	\lambda_{\max}(X)\norm{\psi^0}_2^2 \, + \, L \,\theta_2  \norm{x^1}_2^2.
	\ea
	\non
\eeq
We complete the proof by noting that $\norm{x^{t+1}}_2\le\norm{\psi^t}_2$.

\subsubsection{Proof of Theorem~\ref{thm.main}}
To prove~\eqref{eq.hesUp}, we need to find a feasible solution for $\theta_1$, $\theta_2$ and $X$ in terms of the condition number $\kappa$.  
Let us define
\begin{equation}
\ba{rcl}
X 
&\!\!\!\DefinedAs\!\!\!&
\tbt{x_1 I }{x_0 I}{x_0 I}{x_2 I}
	\, = \;
	x_2
	\tbt{\beta^2 I }{-\beta I}{-\beta I}{I}	
\\[0.35cm]
\theta_2
&\!\!\!\DefinedAs\!\!\!&
\theta_1 (\Lf + \mf) \beta / (1 - \beta) 
\\[0.15cm]
	x_2
	& \!\!\! \DefinedAs \!\!\!&
	( (L + m) \theta_1 \,+\, \theta_2 )/\alpha
	\, = \,
	\theta_2 / (\alpha \beta).
\ea
\label{eq.feasSol}
\end{equation}
If~\eqref{eq.feasSol} holds, it is easy to verify that $X \succeq 0$ with $\lambda_{\min} (X) = 0$, $\lambda_{\max} (X) = (1 + \beta^2) x_2 = \theta_2 (1 + \beta^2) / (\alpha \beta)$, and $A^T X A - X = 0$. Moreover, the matrix $W$ on the left-hand-side of~\eqref{eq.LMI2} is block-diagonal, $W \DefinedAs \diag \, (W_{1}, W_{2})$, and negative semi-definite for all $\alpha \leq 1/L$, where
	\beq
	\ba{rcl}
	W_{1} 
	& \!\!\! = \!\!\! &
	- m (2 \theta_1 L \, C_y^T C_y + \theta_2 \, C_2^T C_2)
	\, \preceq \, 0
	\\[0.15cm]
	W_{2} 
	& \!\!\! = \!\!\! &
	- \left(
	(2 - \alpha (L + m)) \, \theta_1 
	\, + \,
	\alpha (1 - \alpha L) \, \theta_2
	\right) I
	\, \preceq \, 0.
	\ea
	\non
	\eeq

Thus, the choice of $(\theta_1,\theta_2,X)$ in~\eqref{eq.feasSol} satisfies the conditions of Lemma~\ref{lem.generalNester}. Using the expressions for the largest and smallest eigenvalues of the matrix $X$ in equation~\eqref{eq.RHSup} in Lemma~\ref{lem.generalNester}, leads to the upper bound for $\norm{x^t}_2^2$ in~\eqref{eq.hesUp}. Furthermore, from~\eqref{eq.hesUp} we have
	\beq
	\norm{x^t}_2^2
	\, \le \,
	\kappa
	\left( 1 + (1 + \beta^2)/ (\alpha \beta L) \right)
	\norm{\psi^0}_2^2 
	\non
	\eeq
and the upper bound in~\eqref{eq.upperLowerGeneralSCVX} follows from the fact that, for $\alpha$ and $\beta$ in~\eqref{eq.convParam}, 
	$
	1 + (1 + \beta^2)/ (\alpha \beta L)
	= 
	3 + 4/(\kappa - 1).
	$

To obtain the lower bound in~\eqref{eq.upperLowerGeneralSCVX}, we employ our framework for quadratic objective functions in Section~\ref{sec.trans}. In particular, for the parameters $\alpha$ and $\beta$ in~\eqref{eq.convParam}, the largest spectral radius $\rho(A_i)$ corresponds to  $A_n$, which is associated with the smallest eigenvalue $\lambda_n = m$ of $Q$. Since $A_n$ has repeated real eigenvalues $\rho = 1 - 1/\sqrt{\kappa}$, using similar arguments as in Theorem~\ref{thm.sigmax} for quadratic problems we obtain,
	\begin{align*}
		J(t_{\max})
		&\;=\; 
		\sqrt{(t_{\max}-1)^2 \rho^{2t_{\max}} \,+\, t_{\max}^2 \rho^{2(t_{\max}-1)}}
		\\[-0.cm]
		&\;\ge\;
		\sqrt{2} \left(t_{\max}\,-\,1\right) \rho^{t_{\max}}
		\, \ge \,
		{\sqrt{2}(\sqrt{\kappa}-1)^2}/({\mre\sqrt{\kappa}})
	\end{align*}
which completes the proof.

	 \newpage

\end{document}

%% file: commands.tex
\newcommand{\enma}[1]   {\ensuremath{#1}}

\newcommand{\beq}{\begin{equation}}
\newcommand{\eeq}{\end{equation}}
\newcommand{\bseq}{\begin{subequations}}
\newcommand{\eseq}{\end{subequations}}
\newcommand{\beqn}{\begin{eqnarray}}
\newcommand{\eeqn}{\end{eqnarray}}
\newcommand{\ba}{\begin{array}}
\newcommand{\ea}{\end{array}}
\newcommand{\bct}{\begin{center}}
\newcommand{\ect}{\end{center}}
\newcommand{\btmz}{\begin{itemize}}
\newcommand{\etmz}{\end{itemize}}
\newcommand{\benum}{\begin{enumerate}}
\newcommand{\eenum}{\end{enumerate}}

\newcommand{\norm}[1]{\| #1 \|}                 

\newcommand{\diag}      {\enma{\mathrm{diag}}}

\newcommand{\matbegin}{
        \left[
}
\newcommand{\matend}{
        \right]
}

\newcommand{\tbo}[2]{
  \matbegin \begin{array}{c}
       #1 \\ #2
       \end{array} \matend }

\newcommand{\obt}[2]{
  \matbegin \begin{array}{cc}
       #1 & #2
       \end{array} \matend }

\newcommand{\tbt}[4]{
  \matbegin \begin{array}{cc}
       #1 & #2 \\ #3 & #4
       \end{array} \matend }

\newcommand{\be}{\begin{equation}}
\newcommand{\ee}{\end{equation}}

\newcommand{\cplxs}{ C\kern -.35em \rule{0.03 em}{.7 ex}~   }

\def\complex{\hbox{C\kern -.45em \rule{0.03 em}{1.5 ex}}~}

\newcommand{\bi}{\begin{itemize}}
\newcommand{\ei}{\end{itemize}}



%% file: commandsHesam.tex
\newcommand{\R}{\mathbb{R}}
\newcommand{\N}{\mathbb{N}}

\newcommand{\DefinedAs}[0]{\mathrel{\mathop:}=}
\newcommand{\AsDefined}[0]{=\mathrel{\mathop:}}

\DeclareMathOperator*{\minimize}{minimize}

\DeclareMathOperator*{\argmax}{argmax}

\newcommand{\abs}[1]{\lvert #1 \rvert}

\newcommand{\vsp}{\vspace*{0.15cm}}

\newtheorem{mythm}{Theorem}
\newtheorem{myprop}{Proposition}
\newtheorem{mylem}{Lemma}
\newtheorem{myrem}{Remark}

\newcommand{\Lf}{L}
\newcommand{\mf}{m}